\documentclass[11pt]{amsart}
\usepackage{amscd,amsmath,amssymb,amsthm,amsfonts,cite,epsfig,graphics}
\usepackage{graphicx}
\usepackage{epstopdf}
\usepackage[all]{xy}

\newtheorem{theorem}{Theorem}
\newtheorem{lemma}{Lemma}

\newtheorem{proposition}{Proposition}
\newtheorem{example}{Example}
\theoremstyle{definition}
\newtheorem{definition}{Definition}

\theoremstyle{remark}
\newtheorem{remark}{Remark}

\newcommand{\te}{Teich\-m\"ul\-ler}

\newcommand{\C}{\overline{\mathbb C}}
\newcommand{\cp}{\mathbb C}

\newcommand{\Ecupp}{E \cup \{p\}}

\newcommand{\doubleint}{\int\!\!\!  \int}
\newcommand{\w}{w_{t,\epsilon}}

\begin{document}

\title[Guiding isotopies]{Guiding Isotopies and  holomorphic motions}

\author{F. P. Gardiner and Yunping Jiang}
\date{}

\renewcommand{\thefootnote}{}

\renewcommand{\thefootnote}{\arabic{footnote}}

\vspace{.1in}
\begin{abstract}

We develop an isotopy principle for holomorphic motions.
Our main result concerns the extendability of a holomorphic motion
of a finite subset $E$ of a  Riemann surface $Y$ parameterized by a point $t$ in  a pointed hyperbolic surface $(X,t_0).$  If a holomorphic motion from $E$ to  $E_t$ in $Y$ has a guiding
quasiconformal isotopy, then there is a holomorphic extension to any new point $p$ in $Y-E$ that follows the guiding isotopy.
The proof gives a canonical way to replace a quasiconformal motion of the $(n+1)-st$ point by a holomorphic motion while leaving unchanged the given holomorphic motion of the first $n$ points.  In particular, our main result gives a new proof of
Slodkowski's theorem which concerns the special case when the parameter space is the open unit disk with base point $0$
and the dynamical space $Y$ is the Riemann sphere. 
\end{abstract}

\thanks{The second author is partially supported by the collaboration grant from the Simons Foundation (No.199837),
the CUNY collaborative incentive research grant (No.1861), awards from PSC-CUNY, and a grant from the NSF of China (No.11171121),
and a collaboration grant from Academy of Mathematics and Systems Science and the Morningside Center of Mathematics at the Chinese Academy of Sciences.}

\maketitle

\begin{section}{Introduction}~\label{sec1}


Suppose we are given a finite set $\{p_1(t),\ldots,p_n(t)\}$ of holomorphic maps
from a hyperbolic pointed Riemann surface $(X,t_0)$ with values in a Riemann surface $Y$
such that the set $E_t =\{p_1(t),\ldots,p_n(t)\}$ consists of $n$
distinct points in $Y$ for every $t \in X$. Here $X$ is called the parameter space
and $Y$ the dynamical space.  Suppose in addition we are given a continuous function
$X \ni t \mapsto p(t) \in Y$ such that for every value of $t \in X,$ $E_t \cup \{p(t)\}$ consists of $n+1$ distinct points in $Y.$ If all of these hypotheses are satisfied,
then $E_t, t \in X,$ is a holomorphic motion in $Y$ with parameter space $X$
which is extended continuously by the motion  $E_t \cup \{p(t)\}$ in $Y$ with the same parameter space $X$.
The main goal of this paper is to give an additional condition on $p(t)$ that provides
a canonical way to replace the continuous function $p(t): X\rightarrow Y$
with a holomorphic function $\widehat{p} (t): X\rightarrow Y$ such that $\widehat{p}(t) \not= p_{i}(t)$
for all $t\in X$ and $1\leq i \leq n$ and such that $p(t_{0})=\widehat{p}(t_{0}).$ This will give us a holomorphic motion $E_{t}\cup \{\widehat{p}(t)\}$
defined on the same parameter space $X$.

Our main result says that this is possible
as long as we are given a guiding isotopy  for the continuous motion $E_t \cup \{p(t)\}.$
In particular, we give a new proof of Slodkowski's theorem which concerns the case when the pointed parameter space $(X,t_0)$ is the open unit disk $\Delta$ with base point $t_0=0$
and the dynamical Riemann surface $Y$ is the Riemann sphere $\C$ minus three points. The proof has
a  topological part and a geometric part.
The topological part says that any holomorphic motion of a finite subset $E$ of the Riemann sphere with parameter space $(\Delta, 0)$
is always guided  by a guiding quasiconformal isotopy  with the same  parameter space.
The geometric part assumes there is given a guiding isotopy parameterized by the Riemann surface $(X,t_0)$ for a motion in the Riemann surface $Y.$  Assuming this it provides  a canonical holomorphic replacement.  It relies on the existence of a canonical cylindrical differential with a double pole at any point, on the heights mapping for quadratic differentials and on the use of harmonic Beltrami differentials to produce holomorphic coordinates for \te\ space.

The paper is organized as follows. In section 2, we review the definition of a continuous motion and define a guiding quasiconformal isotopy.
In section 3, we state our main results, Theorems~\ref{main} and \ref{main0}, and Slodkowski's theorem. In section 4, we  present the theory of holomorphic quadratic differentials, the heights mapping theorem and a limiting process which
yields a cylindrical quadratic differential with a single semi-infinite cylinder.
In section 5, we define our extension that follows a given guiding isotopy. In section 6, we prove that this extension is holomorphic.
The main ideas in sections 5 and 6  concern extremal infinite cyclinders corresponding to quadratic differentials with double poles and on the use of harmonic Beltrami differentials as coordinates for \te\ space.
In section 7, we give our new proof of Slodkowski's theorem. In section 8 we use Fuchsian groups and their deformations to quasi-Fuchsian groups to generalize the theorem from the case where $Y=\C-\{0,1,\infty\}$ to the case that $Y$ is an arbitrary Riemann surface. To make the paper contain the full proof of Slodkowski's theorem, we also show how to promote holomorphic motions of finite subsets of $Y$ to
holomorphic motions of all of $Y.$  In section 9 we discuss topological obstructions that  show that motions with non simply-connected parameter spaces do not necessarily have guiding isotopies. Thus the guiding isotopy assumption  is necessary.
\end{section}

\vspace*{10pt}
\noindent {\bf Acknowledgement:} We would like to thank Jeremy Kahn, Linda Keen, Sudeb Mitra, Hiroshige Shiga, and Zhe Wang for helpful discussions.
\vspace*{10pt}


\begin{section}{Motions and Guiding Isotopies.}~\label{sec2}


In this section we define motions of a set $E$ in a fixed Riemann surface $Y$ and guiding isotopies of motions.  We assume the motion is parameterized by points $t$ in a pointed Riemann surface $(X,t_0).$

\begin{definition}\label{defn1} Let $E \subset Y.$  A continuous motion $h$ of $E$ in $Y$ parameterized by $X$ with base point $t_0$ is a continuous map $h(t,z): X\times E\to Y$ satisfying
\begin{itemize}
\item[1)] $h(t_{0}, p) =p$, for all $p\in E$ and
\item[2)] for any fixed $t\in X$, $E \ni p\mapsto     h(t,p) \in Y$ is injective.
\end{itemize}
\end{definition}

We think of the parameter $t$ as measuring time and $h_t(z)=h(t,z)$ as specifying the motion of the point $z$ in a dynamical space.

 \begin{definition}\label{defn2}
A continuous motion $h$ is called a holomorphic motion of $E$ if for each $p\in E,$ $X \ni t \to h(t,p) \in Y$ is holomorphic.
\end{definition}

\begin{definition}\label{defn3}
Suppose $h(t,z)$ and  $\hat{h}(t,z)$  are continuous (or holomorphic) motions of subsets $E$ and $\hat{E}$ in $Y.$
If $E \subset \hat{E}$ and $\hat{h}(t,z)=h(t,z)$ for all $z$ in $E,$ then $\hat{h}$ is called
an extension of $h$ to $\hat{E}.$
\end{definition}

\begin{definition}\label{defn4} A guiding isotopy for a continuous motion  $h$ of $E \subset Y$   is an extension of $h$ to a motion  $H:X \times Y \rightarrow Y$ of all of $Y.$ It is called a quasiconformal guiding isotopy if for each $t \in X,$ $Y \ni z \to H(t,z) \in  Y$ is a quasiconformal and if the  Beltrami coefficient $$\mu_t(z) = \frac{\overline{\partial} H(t,z)}{{\partial} H(t,z)}$$ depends continuously on $t$ in the $L_{\infty}$-norm topology.
\end{definition}

We have the following proposition about guiding isotopies.

\begin{proposition}
Suppose $h (t,z): X\times E\to Y$ is a holomorphic motion. If $h$ has a quasiconformal guiding isotopy $H(t,z): X\times Y\to Y$, then any quasiconformal isotopy $G (t,z): X\times Y\to Y$  is isotopic to $H$ on $Y-E.$
\end{proposition}

\begin{proof}
For any $t\in X$, let $H_{t}= H(t,\cdot)$ and $G_{t}=H(t,\cdot)$. Then both of them are quasiconformal homeomorphisms of $Y$. Let $F_{t}= G_{t}^{-1}\circ H_{t}$. Since both of them are extensions of $h$, we have that $F_{t}(z)=z$ for all $z\in E$. Since $F_{t}$ is a quasiconformal homeomorphism, we can define its Beltrami coefficient $\mu_{t} =(F_{t})_{\overline{z}}/(F_{t})_{z}$ on $Y$. It is continuous on $t\in X$.
Thus $\phi (t)=F_{t}$ is a continuous map from $X$ to the space of homeomorphisms of $Y$. But $\phi(t_{0})=F_{t_{0}} =Id$.
\end{proof}

\nocite{SullivanThurston, ManeSadSullivan}

This proposition shows that the isotopy class  of  the extension $H$ of $h$ relative to $Y-E$ is unique.  However, its isotopy class relative to $Y-E \cup \{p\}$ may not be unique.
Our recipe for constructing a holomorphic extension of $h$ to $E \cup \{p\}$ depends only the isotopy class of $H$ relative to $E \cup \{p\}.$

\end{section}

\begin{section}{Statement of the main results}~\label{sec3}
In the following  we let the parameter space be a pointed Riemann surface $(X,t_0)$ and the dynamical space be another hyperbolic Riemann surface $Y.$

\vspace{.1in}

\begin{theorem}[Extending motions of finite sets]~\label{main}
Suppose $$X \times Y \ni (t,z) \mapsto H(t,z) \in Y$$ is a guiding quasiconformal isotopy for a holomorphic motion $p_{j}(t)=h(t,p_j)$ of $n$ distinct points $$E_t=\{p_{1}(t), \cdots, p_{n}(t)\}$$
in $Y.$
Then  for any point $p \in Y$ not equal to any of the points of
$E=\{p_{1}(t_{0}), \ldots, p_{n}(t_{0})\},$ the guiding quasiconformal isotopy $H(t,z)$ determines
a holomorphic motion $\hat{h}(t,p)$ that extends the given holomorphic motion $h(t,z)$ of $E$
to the set $\Ecupp.$ The holomorphically moving point $\hat{h}(t,p)$ depends
only on the guiding isotopy class of $H$ relative to $\Ecupp.$
\end{theorem}

The complete proof of this theorem is given at the end of section \ref{sec8}.

The next theorem can be obtained from this one by letting  $Y$  equal  $\cp_{0,1},$ that is,
the Riemann sphere with $\{0,1,\infty\}$ removed.  However, our approach will be the opposite.  First we prove the next theorem which treats the case when $Y$ is the complement in the Riemann sphere of a set of three or more points.
Then we realize the surfaces $Y, Y_{\epsilon}$ and $Y_{t,\epsilon}$ that come up in the proof
by Fuchsian and Kleinian groups that act on the sphere and, by this means, we obtain a proof of Theorem \ref{main} derived from the the proof of Theorem \ref{main0}.

Suppose $p_{j}(t)=h(t,p_j)$ is a continuous motion of $E=\{p_1,\ldots,p_n\} \subset \C.$ Since the points of $E$ move continuously and distinctly, there is a continuous path of M\"obius transformations $A_t$
such that $A_t(p_1(t)=0, A_t(p_2(t))=1$ and $A_t(p_3(t))= \infty.$ Then  the ordered $n$-tuple
$$A_t(E_t))=\left(0,1,\infty, A_t(p_4(t)),\ldots,A_t(p_n(t))\right)$$ is a continuous motion of $A_0(E).$  We shall say that the motion $A_t \circ h$ is normalized at $0,1$ and $\infty.$

Obviously, if $E_t$ depends holomorphically on $t$ then so does $A_t(E_t)$. Thus by showing
how to extend  a  holomorphic motion of a subset of $\C$ normalized at $0, 1, \infty$, and by choosing an appropriate holomorphic path of M\"obius transformations $A_t$ we also provide a method for extending holomorphic motions that are not normalized.

\vspace*{5pt}
\begin{theorem}[Motions in the Riemann sphere]~\label{main0}
Suppose $$E=\{p_{1}, \cdots, p_{n}\}$$ is a finite subset of $\C$ containing $n \geq 3$ distinct points and
$X$ is a hyperbolic Riemann surface with base point $t_{0}.$   Assume $h: X\times E\to \C$ is a holomorphic motion of $E$ in $\C$ and $p$
is a point in $\C - E.$
If there is a guiding quasiconformal motion $\widetilde{h}: X \times \C \to \C$ that extends $h,$ then
there is a holomorphic motion $\widehat{h}: X\times (E\cup \{p\}) \to \C$ that has the same initial
value at $t=t_0$ as $\widetilde{h}$ and that extends $h$. The holomorphically moving point $\widehat{h}(t,p)$ depends
only on the guiding isotopy class of $\widetilde{h}$ relative to $\Ecupp.$
\end{theorem}

If the parameter space $X$ is compact or compact except for a finite number of punctures, this theorem has no interest because the only holomorphic functions are constant.

The next sections up to and including section 7 are devoted to proving Theorem \ref{main0}. In section  8 we show how to use Theorem \ref{main0} to obtain
Theorem \ref{main}.  The proofs rely on properties of holomorphic
quadratic differentials with double poles, see \cite{Strebelbook}, and on holomorphic coordinates for \te\ space provided by the Bers' embedding, see \cite{Bers6}.

Letting $(X,t_0)$ be the unit disc $\Delta$ with base point $0,$ Theorem \ref{main} becomes the following theorem of Slodkowski.
\vspace*{5pt}
\begin{theorem}[Slodkowski~\cite{Slodkowski}]~\label{slodkowski}
Suppose $X=\Delta=\{ t\;|\; |t|<1\}$ is the open unit disk, with the base point $t_0=0$.
Then any holomorphic motion $h(t,z): \Delta\times E\to \overline{\mathbb{C}}$ can be extended
to a holomorphic motion $\widetilde{h}(t,z): \Delta \times \overline{\mathbb{C}} \to \overline{\mathbb{C}}$.
\end{theorem}

\end{section}

\begin{section}{Cylinders with maximal modulus}~\label{sec4}

\begin{theorem}[Heights Mapping Theorem]~\label{maxmoduli} Assume $E$ is a finite subset of $\overline{\mathbb C}$ containing two or more points and let $\Delta (p)$ be a conformal disc centered at $p$ in  $\overline{\mathbb C} - E.$ Let $\gamma$ be a simple closed curve in $Y=\overline{\mathbb C}-(E\cup \Delta (p))$
and homotopic to the boundary of $\Delta(p).$ Then there exists a unique holomorphic quadratic differential $q$ of finite norm defined on $Y$ with the following properties:
\begin{enumerate}
\item{all of the regular horizontal trajectories     of $q$ are closed and homotopic to      $\gamma,$}
\item{the regular horizontal trajectories form a     cylinder $A$ that fills $\C-(E \cup     \Delta(p))$ except for a critical graph $C,$
}\item{each regular horizontal trajectory $\alpha$ in this cylinder has length equal to $ 2\pi$ and the totality of these trajectories fill $A,$}
\item{$C$ is a closed, connected set of measure     zero and is the union of critical horizontal     trajectories of $q$ that  join its zeros and     poles,}
\item{$C$ is a connected, simply connected finite     graph, the poles of $q$ form the endpoints of     $C$ and any zero of order $k$ is a vertex of     $C$ from which  $k+2$ edges of $C$ emanate,     and}
\item{$||q||=\int \int_{\overline{\mathbb C}-(E     \cup \Delta(p))}|q(z)|dxdy = 2 \pi b,$ where     $b$ is the height of $A$ measured in the     metric $|q|^{1/2}.$}
\end{enumerate}
$q$ is the unique holomorphic quadratic differential on $Y$ with the properties that it has a characteristic cylinder with maximal modulus among all cylinders homotopic to the boundary of $\Delta(p)$  and its circumference measured with respect to the metric $|q(z)|^{1/2}|dz|$  is equal to
$2\pi.$
\end{theorem}

\begin{proof} The existence and uniqueness of $q$ with these properties is well-known, see
\cite{Strebelbook,Jenkins1,Gardinerbook,Gardiner1c}.
For any simple closed curve $\alpha$ not homotopically trivial and not homotopic to a puncture on any hyperbolic Riemann surface, such a holomorphic quadratic differential is obtained by maximizing the modulus of a cylinder among all cylinders on the surface homotopic to $\alpha.$  A similar  conclusion is true even if we begin with a system $\{\alpha_j\}$ of non-homotopic simple closed curves, and either the heights of the annuli can be arbitrarily specified (the heights theorem of Renelt \cite{Renelt}) or the projective class of the moduli of the annuli can be arbitrarily specified (the Strebel moduli theorem \cite{Strebelbook}).

For completeness of exposition here we give a sketch of the proof in the case we need, namely, the case where
there is just one annulus with core curve homotopic to a boundary component and the size of this boundary component shrinking to zero.

\vspace*{5pt}
\begin{lemma}\label{Lvariation}Suppose $A$ is an annular Riemann surface conformal to the ring domain $\{z:r<|z|<R\}$ and $\mu(z)$ is an $L^{\infty}$ Beltrami differential on $A$ with $||\mu||_{\infty}<1.$ In terms of the conformal parameter $w $ let $q(w) (dw)^2= \left(\frac{dz}{z} \right)^2$ and let $A_{\mu}$ be the same
annulus with the conformal structure induced by $\mu.$  This means  a local homeomorphism $w=h(z)$ from a neighborhood of a point $p$ in $A$ is declared to be conformal if
$$ h_{\overline{z}}(z) = \mu(z) h_z(z).$$  Let $\Lambda(A)$ be the extremal length of the family of arcs in $A$ that join its two boundary components and
$\Lambda(A_{\mu})$ be the extremal length with the same family  with respect to the conformal structure induced by $\mu.$  Then

\begin{equation}\label{muvar}
\log \Lambda(A_{t \mu})= \log \Lambda(A) + 2 {\rm \ Re \ } \frac{t}{||q||} \int\!\!\!\int_A \mu q dx dy + O(|t|^2).
\end{equation}
\end{lemma}
\begin{proof} This formula follows from the Reich-Strebel inequalities and is proved in~\cite{Gardiner1b} and~\cite{Gardinerbook}.
\end{proof}

\vspace*{5pt}
\begin{lemma}\label{Lglobalqd} Suppose $A$ is an annulus embedded in a Riemann surface  $Y$ and the modulus of $A$ is as large as possible among all annuli homotopic to $A$ in $Y.$  Let $g$ be a conformal map from $A$ onto the region $\{z: r<|z|<R\}.$  Then
$-\left( \frac{dg}{g}\right)^2$ is the restriction of a holomorphic quadratic differential $q$ on $Y$ and
$$\int\!\!\!\! \int_Y |q| dx dy = \int\!\!\!\! \int_A \left|\frac{1}{ z ^2} \right| dx dy = 2 \pi \log(R/r).$$
All of the regular horizontal trajectories of $q$ are closed and are the images under $g$ of the circles $\rho e^{i \theta}, r<\rho<R.$
\end{lemma}
\begin{proof} This lemma is proved in \cite{Gardiner1c} and, for general measured foliations in \cite{GardinerLakicbook}.  For the benefit of the reader we repeat the main ideas of the proof here.

The complex Banach space $Q(Y)$ of integrable holomorphic quadratic differentials on $Y$ is a closed subspace of  $L^1(Y),$ the space
of integrable quadratic differentials and the dual Banach space $L^1(Y)^*$ is isometric to $L^{\infty},$ the space of essentially bounded Beltrami differentials under the pairing
$$ (q,\mu)= \int\!\!\!\! \int_Y \mu q dx dy.$$
In particular, if $\mu$ represents a linear functional $\ell \in L^1(Y)^*$ then
$$||\ell ||_{*} = \sup_{q \in L^1(Y)}\frac{\left( \int\!\! \int\mu q dx dy
 \right)}{||q||} = ||\mu||_{\infty}.$$
Our strategy to prove this lemma uses the identification of $Q(Y)^{\perp \perp}$ with $Q(Y)$ provided by this pairing.  Showing that $\left(\frac{dg}{g}\right)^2$ is in $Q(Y)^{\perp \perp}$
also shows that it is in $Q(Y).$ If $\left(\frac{dg}{g}\right)^2 \notin Q(Y)^{\perp \perp}$ then there is a Beltrami differential $\mu$ supported in $A$ such that
$$\int\!\!\!\! \int_A \mu q dx dy=0$$
for all $q \in Q(Y)$ and $\int\!\!  \int_A \mu \left(\frac{dg}{g}\right)^2  dx dy \neq 0.$  By the Hamilton-Krushkal variational lemma there would
be a curve of deformations $\mu_t$ tangent to $\mu$ at $t=0$ for which, by Lemma \ref{Lvariation}, $\Lambda(A_{\mu_t})$ is smaller than $\Lambda(A)$ and for which
$Y(\mu_t)$ represents the same point of \te\ space as $Y.$  This conclusion contradicts the assumption that $\Lambda(A)$ has maximum modulus among homotopic annuli $A$ embedded in $Y.$
\end{proof}

The following result is also well known \cite{Gardinerbook, GardinerMasur}.

\vspace*{5pt}
\begin{lemma}\label{annulidistortion} Let $f$ be a quasiconformal homeomorphism mapping a hyperbolic Riemann surface $Y$ to a Riemann surface $f(Y)$ and let $K$ be the maximal dilatation of $f.$  Let $q$ be a holomorphic quadratic differential on $Y$ of finite norm with given heights and $q_f$ a holomorphic quadratic differential on $f(Y)$ such the height along the isotopy class of any curve $\gamma$ in $Y$ measured with respect to $|\!\!{\rm \  Im \ } (q(z)^{1/2} dz)|$ is equal to the height of the isotopy class $f(\gamma)$ in $f(Y)$ measured with respect to $|\!\!{\rm \  Im \ } (q_f(w)^{1/2} dw)|.$  Then
\begin{equation}\label{Dirichlet}
 K^{-1} ||q|| \leq ||q_f|| \leq K  ||q||.\end{equation}
\end{lemma}

\begin{proof}
By the Dirichlet principle \cite{Gardiner4,Gardinerbook} for measured foliations
$$||q||=\int\!\!\!\! \int_Y |q| dx dy$$ is equal to the infimum of $2 D(v) = 2  \int \int_Y (v_x^2 + 2v_y^2) \  dx dy,$
where the infimum is taken over all measured foliations  $|dv|$ that realize the heights of $q,$ and any measured foliation that realizes this infimum is the absolute value of the imaginary part of the square root of $q.$
Note that $|d (v \circ f)|$ has the same corresponding heights on $Y$ that $|dv|$ has on $f(Y),$ and

$$(v \circ f)_z=(v_w \circ f)f_z + (v_{\overline{w}} \circ f)\overline{f}_z.$$

\noindent Since $v$ is real-valued and defined up to plus or minus sign and up to the addition of a constant, $|v_w|$ and $|v_{\overline{w}}|$ are invariant and
$|v_w|^2=
|v_{\overline{w}}|^2=(1/4)(v_x^2+v_y^2).$  Thus we can use exactly the same calculation that is given at the end of Chapter 1 in \cite{Ahlforsbook5}.  We have

$$|(v \circ f)_z| \leq (|v_w| \circ f)(|f_z| + |f_{\overline{z}}|), {\rm \ and \ }$$

$$D_Y (|dv \circ f|)
=2 \int\!\!\! \int_Y |(v \circ f)_z|^2 |dz \wedge \overline{dz}|
\leq $$  $$2 \int\!\!\!  \int_Y (|v_w|\circ f)^2(|f_z|+|f_{\overline{z}}|)^2 |dz \wedge d\overline{z}| =$$

$$2\int\!\!\!  \int_{f(Y)} |v_w|^2
\ \frac{(|f_z|+|f_{\overline{z}}|)^2}{|f_z|^2-|f_{\overline{z}}|^2} \
|dw \wedge d\overline{w}|= $$ $$2 \int\!\!\!  \int_{f(Y)} |v_w|^2
\left(\frac{|f_z|+|f_{\overline{z}}|}{|f_z|-|f_{\overline{z}}|}\right)
|dw \wedge d\overline{w}| \leq K(f) D_{f(Y)}(|dv|).$$
This implies the left hand side of (\ref{Dirichlet}). The right hand side follows from the same argument applied to $f^{-1}.$
\end{proof}

Lemmas \ref{Lvariation}, \ref{Lglobalqd}, and \ref{annulidistortion} complete the proof of Theorem~\ref{maxmoduli}.
\end{proof}

\end{section}

\begin{section}{Definition of the extension}~\label{sec5}

In this section we assume we are given a holomorphic motion $E_{t}=\{ p_{1}(t), \ldots, p_{n}(t)\}$ of a finite set $E=\{p_1,\ldots,p_n\}$ in $\C$ parameterized by $t \in X$.  Without loss of generality, we always assume that $p_1(t)=0, p_2(t)=1$ and $p_3(t)=\infty$ for every $t \in X.$
     Suppose in addition to the holomorphic      functions $p_j(t),$ $1 \leq j \leq n,$ defined      for $t\in X,$ we are given another continuous      function $p(t)$ in $\C$  so that
     $E_t \cup \{p(t)\}=\{p_1(t),\ldots,p_n(t),      p(t)\}$ is a set of $n+1$ distinct points for      each $t\in X.$   If we let \[  h(t,z)=\left\{
     \begin{array}{cl}
     p_j(t) & {\rm \ for \ } z=p_j(t_{0}), 1 \leq j      \leq n {\rm \ and \ } \\
     p(t) & {\rm \ for \ } z=p(t_{0}),
     \end{array} \right. \]
 then $h$ is a motion, which is holomorphic on $E$  and continuous on
 $p=p(t_{0}).$
     Now select a disc $\Delta (p, \epsilon)$ of      radius $\epsilon$ centered at $p$ where      $\epsilon$ is less than the shortest distance      $|p(t_{0})-p_j(t_{0})|, 1 \leq j \leq n.$
     Let $G (t,z)=\widetilde{h}(t,z)$ be the guiding      isotopy in the assumption of      Theorem~\ref{main} such that $p(t)=G(t,p)$.
      Because each map $z \mapsto  G(t,z)$ is       quasiconformal,       $\Delta(t,\epsilon)=G(t,\Delta(p, \epsilon))$       is a quasidisc disjoint from $E_t$ for each       $t\in X.$
Our goal is to use $G$ to define a replacement $\hat{h}$ of $G|(E \cup \{p\})$ by means of a limiting process
in such a way that $\hat{h}$ is a holomorphic motion of $E \cup \{p\}$ that extends the holomorphic motion of $E$ given by $h.$   The extension $\hat{h}$ of $h$ will depend on the equivalence class  relative to $E \cup \{p\}$ of the guiding quasiconformal isotopy $G.$  Also, we want  $\hat{h}$ to have the same initial point as $G$ and in some sense we want it to follow a path topologically
equivalent to the path $t \mapsto p(t).$  We make the definition of $\hat{h}(p)$ in several steps.
\vspace{.1in}

\noindent {\bf Step I.}  For $p \in {\mathbb C}$ let  $\Delta(p, \epsilon)=\{z: |z-p|<\epsilon \}$ and let $j$ be the reflection across the boundary of   $\Delta(p, \epsilon)$
defined by
\begin{equation}\label{reflection}
j(z)=p+\epsilon^2/(\overline{z-p}).
\end{equation}
Let $$Y_{\epsilon}=\C-\left(E\cup j(E)\right).$$
By Theorem \ref{maxmoduli} we construct the cylinder $A_{\epsilon}$ in
$Y_{\epsilon}=\C-\left(E\cup j(E)\right)$ with maximal modulus and with core curve homotopic to the boundary of $\Delta(p, \epsilon)$.  It determines a holomorphic quadratic differential $q_{\epsilon}$ with characteristic cylinder $A_{\epsilon}$ and a conformal map $z \mapsto w$ from that cylinder  onto a domain in the plane of the form
$$\{w:1<|w|<r\},$$
where $(1/2 \pi) \log r $ is the modulus of the cylinder. In this annulus
$$q_{\epsilon}(w)(dw)^2=-\left(\frac{dw}{w}\right)^2$$ and because of the uniqueness of $q_{\epsilon}$ under the symmetry $j,$ the boundary of $\Delta(p,\epsilon)$ is a regular horizontal trajectory of $q_{\epsilon}.$
 Note that the domain
$Y_{\epsilon}=\C - (E \cup j(E))$  is normalized by the condition that the first three points of $E$ are $0, 1$ and $\infty.$  The horizontal foliation with its vertical measure is given by
$$|d \theta| =
|{\rm Re \ } (q_{\epsilon}(w)(dw)^2)^{1/2}|.$$

\vspace{.1in}

\noindent {\bf Step II.}  Let $\tilde{\mu}_{t,\epsilon}$ be the Beltrami coefficient of   $z \mapsto G_t(z)=
G(t,z),$  let

\begin{equation}\label{Bers} \mu_{t,\epsilon} =  \left\{   \begin{array}   {cl}
                                 0 & {\rm \ in \ }      \Delta(p,\epsilon)    \\
                                 \tilde{\mu}_{t,\epsilon}   & {\rm \ in \ }      \C-\Delta(p,\epsilon),
                               \end{array}  \right.
                           \end{equation}
and let $g^{t,\epsilon}$ be the unique quasiconformal self-mapping of $\overline{\mathbb C}$ that fixes $, 1$ and $\infty$  and that has Beltrami coefficient $\mu_{t,\epsilon}.$

\vspace{.1in}

\noindent {\bf Step III.} For each $t,$ the isotopy type of the quasiconformal map
$z \mapsto g^{t,\epsilon}(z)$ determines a heights mapping from $Y_{\epsilon}$ to
$Y_{t,\epsilon}.$ Here we define the quasidisc $\Delta(t,\epsilon)$ by the equation
\begin{equation}\label{delta}
\Delta(t,\epsilon)= g^{t,\epsilon}(\Delta(p,\epsilon)),
\end{equation}
and we put
$Y_{t,\epsilon}$ equal to $g^{t, \epsilon}(Y_{\epsilon}).$
  By the heights mapping theorem, Theorem   \ref{maxmoduli}, there is a unique holomorphic   quadratic differential $q_{t,\epsilon}$ on   $Y_{t,\epsilon}$ such that the heights of
  $$|d \theta_t|=|{\rm Re \ }   (q_{t,\epsilon}(w)(dw)^2)^{1/2}|$$ on   $Y_{t,\epsilon}$ are equal to the corresponding   heights of $|d \theta| = |{\rm  Re \ }   (q_{\epsilon}(w)(dw)^2)^{1/2}|$ on   $Y_{\epsilon}.$
In particular the height of $|d \theta_t|$ along the boundary of $\Delta(t,\epsilon)$ is equal to $2\pi.$

Note that the critical graph of $q_{\epsilon}$ consists of two trees, the first of which has endpoints at the points of $E$ and the second of which has endpoints at the points of $j(E).$  Similarly, the critical graph of $q_{t,\epsilon}$ also consists of two trees and, correspondingly,  the first has endpoints at the points of  $g^{t,\epsilon}(E)$ and the second  has endpoints at the points of $g^{t,\epsilon}(j(E)).$

\vspace{.1in}

\noindent {\bf Step IV.}  By the heights mapping theorem and the minimum Dirichlet principle for measured foliations, \cite{Gardiner4,HubbardMasur},
\begin{equation}\label{htsinequality}(1/K_t)\doubleint_{\C}|q_{\epsilon}|
\leq \doubleint_{\C}|q_{t,\epsilon}|
\leq K_t\doubleint_{\C}|q_{\epsilon}|.\end{equation}
The family of simple closed curves that are homotopic to the boundary of the quasidisc
$\Delta(t,\epsilon)$ in $\C-g^{t,\epsilon}(E \cup j(E))$  has extremal length less than
$$2\pi/  \doubleint_{\C}|q_{t,\epsilon}|.$$

  \noindent {\bf Step V.}
  Since $q_{t,\epsilon}$ is a cylindrical   differential with one cylinder, it is given by
  $$q_{t,\epsilon}(z) =-(1/2\pi)^2   (dw_{t,\epsilon}/w_{t,\epsilon})^2,$$ where   $w_{t,\epsilon}$ is a univalent holomorphic   function of $z$
   in the interior of this cylinder and where $\w$    takes values in $\{w:1/R_{t,\epsilon} < |w| <  1  \}.$  The modulus of the cylinder is $2 \pi/    \log R_{t,\epsilon}.$ We consider the core curve    $C_{t,\epsilon}$ of points in the $z$-plane that    correspond to the points    $|\w|=(1/R_{t,\epsilon})^{1/2}$ in the    $\w$-plane. In the limit as $\epsilon$ approaches zero
 $\w$ converges to a univalent function $w_t$ defined on a maximal punctured disc with range in $\{w: 0<|w|<1\}.$
  \vspace{.2in}

    \noindent {\bf Step VI.}
Since
$\int \int_{\C}|q_{\epsilon}|  \rightarrow \infty$ as $\epsilon \rightarrow 0,$ by
holding $t$ fixed and letting $\epsilon$ approach zero, we see that the annular domain in the $z$-plane lying outside the curve $C_{t,\epsilon}$ has modulus $2\pi/\log R_{t,\epsilon}$ and $R_{t,\epsilon}$ approaches $\infty$ as $\epsilon \rightarrow 0.$ Since the interior of the complement of this annulus
contains at least three distinct points of $E_t,$ by holding $t$ fixed and letting
$\epsilon \rightarrow 0 $ we obtain a punctured disc in the $z$-plane.  By definition, we take $\hat{p}(t)$ to be the puncture of this punctured disc  in the $z$-plane, that is, the point in the
$z$-plane corresponding to the point $w_t=0.$
\begin{definition}\label{defnextension}
We put
 \begin{equation}~\label{defnext}
  \hat{h}(t,z)=
 \left\{
 \begin{array} {ll}
  \hat{p}(t) & {\rm \ for \ } z=p(t_{0}) {\rm \   and}
 \\
 \lim_{\epsilon \rightarrow 0} g^{t,\epsilon}(z) &  {\rm \ for \ } z = p_j(t_{0}), 1 \leq j \leq n.
 \end{array} \right.
\end{equation}

\end{definition}

\begin{theorem}\label{extension} The function
$\hat{h}(t,z)$
is a motion of $E \cup \{p\}$ that extends  the holomorphic motion $h$ of $E.$ That is,  for each $t$ the extension $\hat{h}$ of $h$ is an injection from $E \cup \{p\}$ into $\overline{\mathbb C}.$   \end{theorem}
\begin{proof} The map $g^{t,\epsilon}$ defined in Step II depends on the reflection defined in Step I, which depends on $\epsilon.$  As $\epsilon \rightarrow 0,$ the Beltrami coefficients  $\mu_{t,\epsilon}$ converge in the bounded pointwise sense to the Beltrami coefficient of $z \mapsto G(t,z)$ and so since these mappings are normalized to fix $0, 1$ and $\infty,$  $g^{t,\epsilon}$ converges uniformly in the spherical metric to $z \mapsto G(t,z).$
Therefore, $\hat{h}(p_j(t))=h(p_j(t))$
for $p_j(t) \in E_t.$
Since $\hat{p}(t)=\hat{h}(p(t))$ lies at the center of a punctured disc that excludes all the points of $E_t,$  it cannot coincide with any of those points.
 \end{proof}
The limiting differential $q_t=\lim_{\epsilon \rightarrow 0}q_{t,\epsilon}$ is a holomorphic on the Riemann sphere except for possibly simple poles at $E_t$ and a double pole with quadratic residue equal to $-2\pi$ at $\hat{p}(t).$
  In the next section we will show that $\hat{h}$   extends  $h$ holomorphically to the point $p.$
  \end{section}

\vspace{.2in}

\begin{section}{Harmonic coordinates for \te\ space}~\label{sec6}

In this section  we use harmonic Beltrami coefficients as coordinates for the unreduced \te\ space of
the bordered Riemann surface $R=Y_{\epsilon} - \Delta(p,\epsilon).$  This \te\ space consists of quasiconformal maps $f$ from $R$ to $f(R)$ factored by an equivalence relation.  Two such maps
$f_0$ and $f_1$ are equivalent if, after postcomposing one of them by a conformal map, they can be connected by an isotopy $g_t$ which pins down the points of $E$ and all of the points on the boundary of $\Delta(p,\epsilon).$  In particular, the Teichm\"uller space is infinite dimensional and its coordinates  at the point $[f]$ keep track of the boundary curve of $R,$ namely, the boundary of the quasidisc $f(\Delta(p,\epsilon)).$

From the Ahlfors-Weill extension procedure (see~\cite{AhlforsWeill}) about the existence of a local holomorphic section,
a local coordinate at a point $\tau=[f]$  in $T(R)$ is given by harmonic Beltrami coeffients $\mu$ on $f(R),$ which have the following special properties:
1) $\mu$ is identically equal to zero on $f(\Delta (p,\epsilon)),$

2) $\mu = \rho^{-2}(z)\overline{\psi(z)}$ where $\psi$ is a holomorphic quadratic differential on
$f(R)$ where $\rho$ is the Poincar\'e metric for $f(R),$ and

3) where $||\mu||_{\infty}<1.$

In particular, the holomorphic function $\psi(z)$ has at most simple poles at the points of $f(E)$  and
3) means that $|\psi(z)| \leq C_{\psi} \delta^{-2}(z)$ where $\delta(z)$ is the minimum Euclidean distance from $z$ to the boundary of $f(\Delta(p,\epsilon)).$

\begin{theorem}\label{holtheorem}
Assuming each of the functions $p_j(t), 1 \leq j \leq n$ is holomorphic, the motion
$\hat{h}(p,t)$ defined in (\ref{defnext}) holomorphically extends
the motion $h(p,t).$ In particular,
\begin{itemize}
\item[a)] $\hat{h}(t,z)$ coincides with $h(t,z)$     for each $z=p_j(t), 1 \leq j \leq n,$
\item[b)] $\hat{p}(t)=\hat{h}(t,p)$ is distinct     from every point of $E_t$ in $\C$ for each $t$     in $X,$ and
\item[c)] $\hat{p}(t)$ is a holomorphic function     of $t \in X.$
\end{itemize}
\end{theorem}

\begin{proof}
Part a) was proved in the previous section and part b)
 follows because $\hat{h}(t,p)$ is separated from  the points of $E_t$ in $\C$ by  annuli of  arbitraritly large moduli.

To prove c) consider the Jordan domain $\Delta_{t_0,\epsilon}$ consisting of those points that lie inside the Jordan curve $C_{t_0,\epsilon}$ defined in Step V of the previous section.  For $|t-t_0|<\delta,$ where $\delta$ is sufficiently small, all of the points $p_1(t),\ldots,p_n(t)$ lie outside this domain.  
 In Step V we have chosen the parameters $w_{t,\epsilon}$ so that they converge uniformly to a univalent function $w_t$ in the domain which is the exterior of the curves $C_{t_0,\epsilon}$ in the $z$ plane parameter, and in the limit this domain corresponds to the region $|w_t|<1.$  In the $w_{t,\epsilon}$ parameter the quasi-circle $C_{t_0,\epsilon}$ such that the extremal length of the family of  curves homotopic to $C_{t_0,\epsilon}$ approaches zero as $\epsilon$ approaches zero and $w_{t,\epsilon}$ converges uniformly on compact sets to $w_t.$
 Consider  the \te\ space $Teich(E_{t_0} \cup j(E)_{t_0}).$ Because the guiding homeomorphism is assumed to be continuous, there is a positive number $\delta$ such that for $|t-t_0| <\delta,$ the holomorphically motion of points from 
 $\{p_1(t_0),\ldots,p_n(t_0)\}$ to $\{p_1(t),\ldots,p_n(t)\}$ is represented by homorphically moving harmonic Beltrami coefficents defined in the $w_{t,\epsilon}$ parameter supported in the interior fo the quasicircle $C_{t_0,\epsilon}.$ Since the their supports are in this interior, all of the points in the exterior move holomorphically with respect to the parameter $t.$ But $w_t$ is the uniform limit of $w_{t,\epsilon}$ as $\epsilon \rightarrow 0.$ This implies that $\hat{p}(t)$ defined in Step VI is a holomorphic function of $t.$
 \end{proof}
\end{section}

\begin{section}{A new proof of Slodkowski's Theorem}~\label{sec7}

In this section, we give a new proof of Theorem~\ref{slodkowski}. It is based on two results, the first a known topological result
and the second a geometric result, which is our main Theorem \ref{main}.  Therefore, to complete our new proof of Theorem~\ref{slodkowski},
we only need to prove the topological result here by showing that the topological condition in Theorem~\ref{main} holds when the parameter space $X$ is the open unit disk $\Delta$. This is Lemma \ref{qcmotions}  below.

To state  the lemma we introduce some notation. Let $M(\cp)$ be the open unit ball in
$L^{\infty} (\cp)$. An element $\mu$ in $M(\cp)$ is called a Beltrami coefficient. The measurable Riemann mapping theorem \cite{Ahlforsbook5} says
for each $\mu \in M(\cp)$, the Beltrami equation
$$
w_{\overline{z}} =\mu (z) w_{z}
$$
has a unique solution fixing $0$, $1$, and $\infty$. The unique normalized solution is denoted by $w^{\mu}$.
Then $w^{\mu}(z)$ depends on $\mu$ holomorphically. For any closed subset $E\subset \C$,
we say two elements $\mu_0$ and $\mu_1$ in $M(\cp)$ are \te\ equivalent if there is a continuous curve of Beltrami coefficients $\mu_t$ coinciding with $\mu_0$ and $\mu_1$ at $t=0$ and $t=1$ such that
the normalize quasiconformal maps $w^{\mu_t}$ fixing the same three points of $E$ also fix all the other points of $E$ for $0 \leq t \leq.$
 We use $[\mu]_{E}$ to denote an equivalence class.  Let $\Omega=\C\setminus E$.
From a theorem of Earle and McMullen \cite{EarleMcMullen}, two quasicconformal deformations $w^{\mu}$ and $w^{\nu}$ of $\Omega$ are \te\ equivalent if they are connected by a quasiconformal homotopy rel $E$.
 $T(E)$ is a contractible complex Banach manifold.

By a theorem of Lieb \cite{Lieb} it is known that when $E$ has measure zero, $T(E)$ and $Teich(\Omega)$ are isomorphic.
The complex manifold structure on $T(E)$ is the finest structure
such that the projection $P_{E} (\mu) =[\mu]_{E}$ is a holomorphic map.

\vspace*{5pt}
\begin{lemma}~\label{qcmotions}
Suppose $\Delta$ is the open unit disk with the base point $0$. Suppose $E=\{p_{1}=0,p_{2}=1,p_{3}=\infty, p_{4},\dots, p_{n}\}$, $\#(E)=n\geq 3$,
is a finite subset of the Riemann sphere $\overline{\mathbb C}$ and
$h(t,z): \Delta\times E\to \overline{\mathbb C}$ is a normalized holomorphic motion of $E.$
Then $h$ has a quasiconformal guiding isotopy $H(t,z): \Delta\times \overline{\mathbb C}\to \overline{\mathbb C}$.
\end{lemma}

\begin{proof}
Let $\Omega=\C\setminus E$ be the Riemann sphere punctured at the points of $E.$
The Teichm\"uller space $T(E)$ is the classical Teichm\"uller space of the plane domain $\Omega.$
Let
$$
Y_{n-3}=\{ {\bf z} =(z_{1}, \cdots, z_{n-3})\in \mathbb{C}^{n-3}\}
$$
where $z_{i}\not=z_{j}$ for all $1\leq i\not=j\leq n-3$ and $z_{i}\not=0, 1, \infty$ for all $1\leq i\leq n-3$.
From the normalized holomorphic motion $h(t,z): \Delta\times E\to \C$, we can define a holomorphic map
$$
f(t) =(h(t, p_{4}), \cdots, h(t,p_{n})): \Delta\to Y_{n-3}.
$$
From~\cite{BersRoyden} (see also~\cite{Nag}), we know that
the map
$$
\pi_{E} ([\mu]_{E}) = (w^{\mu}(p_{4}), \cdots, w^{\mu}(p_{n})) : T(E)\rightarrow Y_{n-3}
$$
is a holomorphic universal covering.
Since $\Delta$ is simply connected, we can lift $f$ to get a holomorphic map
$$
\widetilde{f}(t): \Delta\to T(E)
$$
such that
$$
\pi_{E} \circ \widetilde{f}= f.
$$

Let $P_E:M({\mathbb C}) \rightarrow T(E)$ be the natural projection from a Beltrami coefficient $\mu$ to the \te\ class in $T(E)$ of the quasiconformal mapping $w^{\mu}.$
From the Douady-Earle barycentric extension procedure (see \cite{DouadyEarle}), there is a continuous section
$S$ of $P_{E}$ (see~\cite{JiangMitra2}), that is, a continuous map $S$ from $T(E)$ to $M(\cp)$ such that
$P_{E} \circ S$ is the identity on $T(E).$
Define
$$
\widehat{f} (t) = S\circ \widetilde{f} (t) : \Delta \to M(\cp)
$$
Then we have that
$$
P_{E}\circ \widehat{f} =\widetilde{f}\quad \hbox{and}\quad \pi_{E}\circ P_{E}\circ \widehat{f} =f.
$$
Refer to the following diagram:\\
\vspace*{5pt}
\centerline{
\xymatrix{
          &  M ({\mathbb C}) \ar[d]^{P_{E}} \\
          & T(E) \ar[d]^{\pi_{E}}           \ar@/_1.3pc/[u]_{S} \\
\Delta \ar[uur]^{\widehat{f}} \ar[ur]^{\widetilde{f}}  \ar[r]^{f} & Y_{n-3} } }

For any $\widehat{f}(t)\in M(\cp)$, let $w^{\widehat{f}(t)}$ be the normalized
quasiconformal homeomorphism solving the Beltrami equation
\begin{equation}~\label{be}
w_{\overline{z}}= \mu (z) w_{z}
\end{equation}
with the Beltrami coefficient $\mu=\widehat{f}(t)$.
Define
$$
H(t,z) =w^{\widehat{f}(t)} (z): \Delta\times \C\to \C.
$$
Since $\pi_{E}\circ P_{E} (\widetilde{f}) =f$,
we get $H(t, 0)=0$, $H(t,1)=1$, $H(t, \infty)=\infty$, and
$$
(H(t, p_{4}), \cdots, H(t, p_{n}))=f(t).
$$
Thus $H$ is an extension of $h$. Thus $H$ is a quasiconformal guiding isotopy for $h$.
(Actually, $H$ is a normalized quasiconformal motion of $\C$ with parameter space $\Delta$ which extends $h$).
This completes the proof.
\end{proof}

Lemma~\ref{qcmotions} says that for any normalized holomorphic motion $h: \Delta\times E\to \C$ of any finite subset $E$
of the Riemann sphere $\C$ with parameter space $\Delta$,
our guiding isotopy assumption in Theorem~\ref{main} holds.
Thus for any new point $p\not\in E$, we can have a holomorphic motion $\widehat{h}: \Delta \times (E\cup\{p\})\to \C$ extending $h$.
To complete the proof, we need the $\lambda$-Lemma of Ma\~n\'e, Sad and Sullivan, \cite{ManeSadSullivan}.

\vspace*{5pt}
\begin{lemma}[$\lambda$-Lemma]~\label{mss}
Suppose $h(t,z): \Delta\times E\to \C$ is a holomorphic motion, where $E$ is a (not necessarily finite)  subset of $\C$.
Then it can be extended uniquely to a holomorphic motion $\overline{h} (t,z): \Delta\times \overline{E}\to \C$, where $\overline{E}$ means the closure of $E$ in $\C$.
\end{lemma}

Now suppose $h(t,z): \Delta\times E\to \C$ is the normalized holomorphic motion in Theorem~\ref{slodkowski}. Let $E_{\infty}=\{ 0, 1, \infty, p_{1}, \cdots, p_{n},\cdots\}$ be a countable dense subset of $E$.
Let $F=\{ q_{1}, \cdots, q_{n}, \cdots\}$ be a countable dense subset of $\C\setminus E$. Let $E_{n}=\{ 0, 1, \infty, p_{1}, \cdots, p_{n}\}$ and $F_{n} =\{q_{1}, \cdots, q_{n}\}$.
Then
$h_{n}= h|\Delta\times E_{n}$ is a holomorphic motion for every $n>3$. Our main result (Theorem~\ref{main}) with the consideration of Lemma~\ref{qcmotions} implies that we can extend $h_{n}$ to a
holomorphic motion $H_{n}(t,z): \Delta\times (E_{n}\cup F_{n})\to \C$. Inductively, we have a holomorphic motion $H_{\infty} (t,z): \Delta\times (E_{\infty}\cup F)\to \C$ which extends every $h_{n}$.
The $\lambda$-Lemma implies we can extend this last holomorphic motion into a holomorphic motion $H(t,z)$ of the closure of $E_{\infty}\cup F$ which is the whole Riemann sphere $\C$ with parameter space $\Delta$.
This holomorphic motion $H$ is an extension of $h$. This completes our new proof of Theorem~\ref{slodkowski}.

\end{section}

\begin{section}{Arbitrary dynamical spaces}\label{sec8}

In this section we show how the proof given in the previous sections for the case when the dynamical space is the Riemann sphere can be generalized to prove the analogous theorem when the dynamical space is any Riemann surface.  This is the content of Theorem~\ref{main}.

 We assume that the finite subset $E$ of $Y$ is  moving holomorphically and parameterized by a  guiding quasiconformal isotopy with parameter $t  \in X.$  Our goal is to use the isotopy to induce a  holomorphically moving point $p(t)$ that starts at  any point $p=p(t_{0})$ in $Y-E$ and has the  property that $p(t)$ never lies in the set $E_t.$

 First, take a point $p \in Y$ and form a Fuchsian  universal covering  group $\Gamma$ of $Y-\{p\}$  acting on the
 upper half plane ${\mathbb H}$ that contains a  primitive parabolic element $\gamma_0(z)= z+1$ with  the property that the arc $\tilde{\alpha}(t) =t  +iM, 0 \leq t \leq 1$ has  image $\alpha$ under the  covering that winds once around the point $p.$
By the Leutbecher-Shimizu \cite{Leutbecher,Shimizu} lemma the restriction of the covering mapping to the semi-infinite strip $$\{w: 0 \leq u= {\rm Re \ } w <1, v = {\rm Im \ }w>1\}$$ is a homeomorphism onto an embedded punctured disc contained in $Y-\{p\}$ with the puncture located at $p.$
Thus, $z=\exp(2 \pi i w)$ is a local parameter on $Y$ with $z=0$ corresponding to the point $p$
and
the part of this strip lying above $v=M>1$ corresponding to an embedded disc
\begin{equation}\label{smalldisc}\{z:|z| \leq \epsilon =\exp (-2 \pi M)\}.\end{equation}

Now we can replace Step I of section \ref{sec6} by the following device.  We construct a Kleinian group $\Gamma_{\epsilon}$  that contains two isomorphic copies $\Gamma$ both of which  act discontinuously on a certain domain.  The domain has an anticonformal involution $j$ and the action of one copy of $\Gamma$
covers $Y_{\epsilon}$ and the action of $j \circ \Gamma \circ j$ on the other copy covers $j(Y_{\epsilon}).$  It turns out that $\Gamma_{\epsilon}$ is called an HNN-extension of $\Gamma,$
named after Higman, Neumann and Neumann.

  To construct $\Gamma_{\epsilon},$ we take $M>1$
  and let $j$ be reflection in the plane around the   horizontal line $y=M.$
  Then we form  Kleinian group $\Gamma^d$ generated   by $\Gamma$ and $j \circ \Gamma \circ j.$
  Note that both of these groups contain the common   element $z \mapsto z+1.$
   If we take a fundamental domain for $\Gamma$ that    is bounded by the two vertical lines $x=0$ and
  $x=1$ and axes of hyperbolic transformations all   lying below the line $y=1,$ then $\Gamma^d$ has a   fundamental domain lying between the two   horizontal lines $y=0$ and $y=2M.$ This rectangle   contains the subrectangle $0 \leq x \leq 1$ and $1  \leq y \leq 2M-1,$  which is entirely contained   within the fundamental domain for $\Gamma^d.$    Moreover,  the fundamental domain is invariant   under its reflection $j$ across the line $y=M.$     $\Gamma^d$ is the amalgamated product of $\Gamma$   and $j \circ \Gamma \circ j,$ amalgamated along   the common cyclic subgroup generated by $z \mapsto   z+1.$
  Since $\Gamma^d$ depends on $M$ and $M$ depends on   $\epsilon$ by the simple formula   (\ref{smalldisc}), we denote $\Gamma^d$ by   $\Gamma_{\epsilon}.$
  The groups $\Gamma$ and $\Gamma_{\epsilon}$ both   act on the Riemann sphere and, if we hold the   limit points $0,1,$ and $\infty$ fixed
  and let $\epsilon \rightarrow 0,$ the fundamental   domains for the groups $\Gamma_{\epsilon}$
  converge in the Hausdorff sense to the fundamental   domain for $\Gamma.$

\vspace{.2in}

\vspace*{10pt}

\setlength{\unitlength}{.5in}
\begin{picture}(10,5)(-4.6,0)
\linethickness{0.5pt}
\put(-5,0){\vector(1,0){10}}
\put(0,0){\vector(0,1){5}}
\put(-5,0.5){\line(1,0){10}}
\put(-4,0.6){\makebox(0,0){$y=1$}}
\put(-5,2){\line(1,0){10}}
\put(-4,2.1){\makebox(0,0){$y=M$}}
\put(0.2,2.2){\makebox(0,0){$\tilde{\alpha} (t)$}}
\put(1,2.2){\makebox(0,0){$\uparrow$ $i$}}
\put(3,3){\makebox(0,0){$\Gamma_{\epsilon}\rightarrow \Gamma$ as $\epsilon \to 0$ or $M\to \infty$}}
\put(2,4.5){\makebox(0,0){$z\rightarrow z+1$}}
\put(-5,4){\line(1,0){10}}
\put(-4,4.1){\makebox(0,0){$y=2M$}}
\put(-5,3.5){\line(1,0){10}}
\put(-4,3.6){\makebox(0,0){$y=2M-1$}}
\put(0.5,0){\line(0,1){5}}
\put(0.9,0.8){\makebox(0,0){$x=1$}}
\put(-0.5,0.8){\makebox(0,0){$x=0$}}
\linethickness{1pt}
\put(0,2){\line(1,0){0.5}}
\end{picture}

\vspace{10pt}

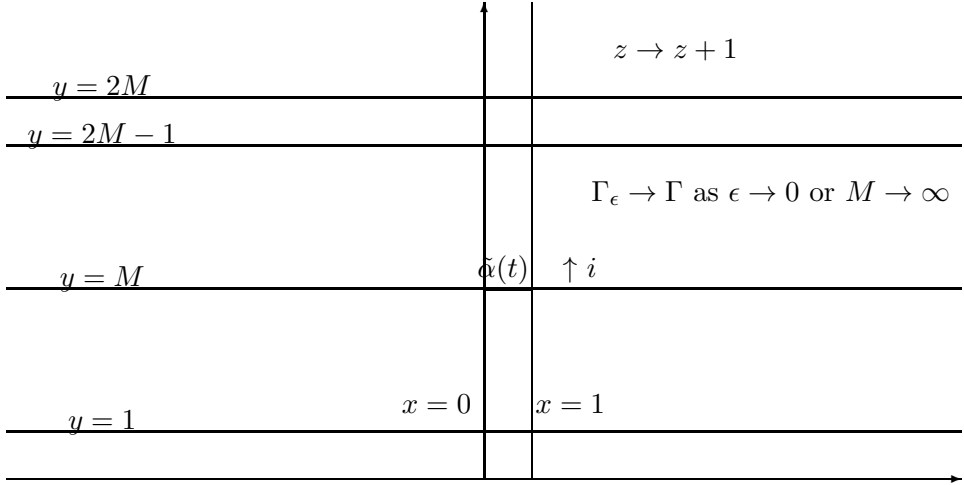
\centerline{Figure 1}

\vspace{.3in}
Steps II, III and IV can be repeated almost verbatim but now with Beltrami coefficients $\tilde{\mu}_{t,\epsilon}$ and $\mu_{t,\epsilon}$ and corresponding mappings $g^t$ and $g^{t,\epsilon}$   all equivariant for the Kleinian groups $\Gamma_{\epsilon}.$  As one lets $M \rightarrow \infty$ or $\epsilon\rightarrow 0$ and holds three limit points $0,1$ and $\infty$ fixed, the groups $\Gamma_{\epsilon}$ converge to $\Gamma,$  and so do their corresponding conjugated groups $g^{t,\epsilon} \circ \Gamma_{\epsilon} \circ (g^{t,\epsilon})^{-1}$ converge to $g^{t} \circ \Gamma \circ (g^{t})^{-1}$.  The extension of the holomorphic motion to $p$ is realized by the  corresponding fixed point of a parabolic transformaton in  $g^{t} \circ \Gamma \circ (g^{t})^{-1}.$ Since the entries of this parabolic transformation vary holomorphically with $t,$ so does the corresponding fixed point $\hat{p}(t).$ This completes the proof of Theorem~\ref{main}.

\end{section}

\begin{section}{Trace monodromy and the isotopy  principle}\label{sec9}

In this section we show why the guiding isotopy  assumption in Theorem~\ref{main} is necessary.
To do this we describe a topological obstruction to the extension of continuous motions parameterized by a surface with non-trivial fundamental group. This description depends on monodromy and something we call trace monodromy,  concepts which are developed in
~\cite{BeckJiangMitraShiga}.
Using these ideas we give an example of a holomorphic motion
of a finite subset in  Riemann sphere $X$ parameterized by any non-simply connected bounded domain $Y$ in the complex plane that cannot be extended to
a continuous  motion of the whole Riemann sphere parameterized by the same domain.

Suppose $X$ is any hyperbolic Riemann surface and $Y=\C-E$ where $E=\{0,1,\infty, p_4,\ldots ,p_n\}.$
If $h$ is a continuous motion of $ X \times E\to \C$, then for any $z$ in $E-\{0,1,\infty\}$,
$h(t, z)=h^{z}(t): X\to \mathbb{C}_{0,1}=\overline{\mathbb{C}}\setminus \{0,1,\infty\}$ is continuous and, for any choice of $z \in E-\{0,1,\infty\}$ one obtains a different homomorphism of fundamental groups:
$$
h_{*}^{z}: \pi_{1} (X)\to \pi_{1} (\mathbb{C}_{0,1}).
$$We call these homomorphisms the {\it trace monodromies} induced by $h$.
Taking into account the arbitrary normalization at three points  in $E$ and the arbitrary choice of a fourth, one finds that  the number of  different trace monodromy conditions is given by the binomial coefficient
$$\left(\begin{array}{c}
n\\
4
\end{array}\right).
$$

By definition a  trace monodromy is
trivial if it maps every element of $\pi_{1}(X)$ to the
identity of $\pi_{1}(\mathbb{C}_{0,1})$. The trace monodromy obstruction to topological extension described in the following theorem and the more general monodromy obstruction are presented in~\cite{BeckJiangMitraShiga}.

\vspace*{5pt}
\begin{theorem}~\label{BeckJiangMitraShiga}
Suppose $X$ is a Riemann surface with a base point $t_{0}$.
Let $h: X\times E\to \overline{\mathbb{C}}$ be a normalized holomorphic motion of a
    finite set $E$ with $card(E)\geq 4.$  If $h$ has a quasiconformal guiding isotopy,    then for  each   $z \in E$
 the trace monodromy $h^{z}_{*}:     \pi_{1}(X,t_0)\to \pi_{1}(\mathbb{C}_{0,1})$ is trivial.
\end{theorem}

\begin{proof}
We outline the proof here.
Suppose $H(t,z): X\times \C\to \C$ is a continuous motion which extends $h(t,z): X\times E\to \C$. Let $\gamma (s)$, $0\leq s\leq 1$,
be a simple closed curve in $X$ with $\gamma (0)=\gamma (1)=t_{0}$. Let $\mu (s)$ be the Beltrami coefficient of $H (\gamma (s),\cdot)$ which is continuous on $s$.
Then $H(\gamma (s), \cdot) =w^{\mu (s)} (\cdot)$ since both are quasiconformal maps fixing $0,1, \infty$ with the same Beltrami coefficient.
Let $\hat{H}(s, t) = w^{t\mu (s)} (z): [0,1]\times [0,1]\to \cp_{0,1}$ for any $z\not=0,1, \infty\in E$.
Then it is a continuous map such that $\hat{H}(s,1) =H(\gamma (s), z)$ and $\hat{H}(s, 0)=z$.
Thus $\hat{H}(\gamma (s), z)$ is a continuous curve in $\cp_{0,1}$ homotopic to a point $z$ in $\cp_{0,1}$. This implies that the trace monodromy is trivial.
\end{proof}

Using this theorem, we can construct a counterexample of a holomorphic motion of a finite subset of
the Riemann sphere parameterized by any non-simply connected planar domain which does not satisfy our isotopy class assumption in Theorem~\ref{main}.

\vspace*{5pt}
\begin{example}~\label{ex}
Suppose $X$ is a planar domain in the Riemann sphere $\C$ such that $\C - X$ has more than one connected component and at least one component contains more than one point.
Let $t_{0}$ be the base point of $X.$  Then for any finite subset $E$ in $\C$ with $\#(E)\geq 4$,
there is a holomorphic motion $h(t,z): X\times E\to \C$ that
has no quasiconformal guiding isotopy.
\end{example}

\begin{proof}
Since $\C-X$ has a component containing more than one point, we can use two points in this component and a square root map to map $X$ into a half-plane.
Then applying a M\"obius transformation, we can assume that $X$ is a bounded planar domain such that $\C-X$ has one unbounded component and several bounded components.

Suppose $z_{0} \in E-\{0, 1, \infty\}.$ Since $X$ is planar, we can map it conformally  to a planar domain $\widetilde{X}$ containing only one point $z_{0}$ in $E$. Thus we have a domain
$\widetilde{X}$ such that $z_{0}\in \widetilde{X}$ and
$\widetilde{X}\cap (E- \{z_{0}\})=\emptyset$ and $0$ is in a bounded component of $\C - \widetilde{X}$ and $E-\{0, z_{0}\}$ are all in the unbounded component of $\C - \widetilde{X}$ and a conformal map
$z=\phi(t): X\to\widetilde{X}$
such that $\phi (t_{0}) =z_{0}$.
Define $h(t, z)=z$ for any $z\not=z_{0}$ and $t\in X$ and $\phi (t,z_{0})=\phi (t)$. Then $h$  is a holomorphic motion with non-trivial trace monodromy.
From Theorem~\ref{BeckJiangMitraShiga}, it does not have a quasiconformal guiding isotopy.
\end{proof}

\vspace{.2in}

\begin{remark}
When the cardinality of  $E$ is $4$ and $X={\mathbb C}-\{0,1\}$ is the thrice-punctured sphere with a base point $t_{0}$, Douady constructed the following counterexample.
Let $E=\{ 0, 1, \infty, t_{0}\}$, let $h(t,z): X\times E\to \C$ and define $h$ by $h(t,0)=0$, $h(t,1)=1$, and $h(t, \infty)=\infty$, and $h(t,t_{0})=t$.
Then Douady showed that $h$ is a maximal holomorphic motion and, therefore, cannot be extended further.
Since an annulus $A$ can be thought as a covering space of the thrice-punctured sphere, there is a covering map $\pi: A\to X$.
Earle considered $\widetilde{h} (t,z)=(\pi^{*}h) (t,z)=h(\pi(t), z): A\times E\to \C$ and showed it is a maximal holomorphic motion and so it also cannot be extended further. See ~\cite{Earle2} for these two counterexamples and the definition of a maximal holomorphic motion.
The topological obstruction defined in~\cite{BeckJiangMitraShiga} gives us more flexibility to construct more counterexamples.
One can find other counterexamples when the parameter space is the punctured disk in~\cite{BeckJiangMitraShiga} or an annulus.
The counterexample given in Example~\ref{ex} works for any finite subset with any non-simply connected planar domain.
\end{remark}

\vspace{.1in}

\begin{remark}
In~\cite{BeckJiangMitraShiga} it is shown that when $card(E)=4$ a holomorphic motion $h: X\times E\to \C$ can be extended to a holomorphic motion
$\widetilde{h}: X\times \C\to \C$ if, and only if, its trace monodromy is trivial.
\end{remark}
\end{section}

\bibliographystyle{plain}
\bibliography{Articles,books}

\vspace{.2in}

\noindent {Frederick P. Gardiner,} Department of Mathematics, Brooklyn College, Brooklyn, NY 11210 and Department of Mathematics, CUNY Graduate Center,  89365 Fifth Avenue, New York, NY 10016. Email:frederick.gardiner@gmail.com

\vspace{.1in}

\noindent {Yunping Jiang,} Department of Mathematices, Queens College CUNY, Flushing,  NY 11367 and Department of Mathematics, CUNY Graduate Center,  365 Fifth Avenue, New York, NY 10016. Email: yunping.jiang@qc.cuny.edu

\end{document}